\documentclass[11pt]{amsart}

\usepackage{amsmath, amsfonts, amssymb,amsthm}
\usepackage{amstext}
\usepackage{mathrsfs}

\usepackage{float,epsf,subfigure}
\usepackage[all,cmtip]{xy}
\usepackage{hyperref}
\usepackage{algorithm}
\usepackage{algorithmic}
\usepackage{mathtools}
\usepackage{float}
\usepackage{bm}
\theoremstyle{plain}
\newtheorem{theorem}{Theorem}[section]

\newtheorem{cor}[theorem]{Corollary}

\newtheorem{def-thm}[theorem]{Definition-Theorem}
\newtheorem{lemma}[theorem]{Lemma}

\newtheorem*{tha}{Theorem A}
\newtheorem*{thb}{Theorem B}
\newtheorem*{thc}{Theorem C}

\theoremstyle{definition}

\begin{document}
\title[On  Picard's  Problem via Nevanlinna Theory]{On  Picard's  Problem via Nevanlinna Theory}
\author[X.-J. Dong]
{Xianjing Dong}

\address{School of Mathematical Sciences \\ Qufu Normal University \\ Qufu, Jining, Shandong, 273165, P. R. China}
\email{xjdong05@126.com}


\date{}
\maketitle \thispagestyle{empty} \setcounter{page}{1}

\begin{abstract}   We  consider   the classical  Picard's problem  for  non-parabolic  complete K\"ahler manifolds  with non-negative Ricci curvature. 
Based on the global Green function approach, we give   a positive answer to   Picard's problem under certain  condition  
by developing   Nevanlinna theory. That is,   we prove   that every meromorphic function on such a manifold 
  reduces to   a constant if it omits three distinct values, provided that the manifold satisfies a  volume growth condition; and  prove  that every meromorphic function of non-polynomial type growth  on such a manifold can avoid 2 distinct values at most. 
   \end{abstract}

\setlength\arraycolsep{2pt}
\medskip

\vskip\baselineskip

\section{Introduction}

\subsection{Motivation}~

The famous  Picard's  theorem asserts  that 
a meromorphic function  on the complex Euclidean spaces must reduce to a constant if it omits three distinct values.  
Historically,  many authors  have been committed  to extensions of    this theorem for a long time (see \cite{Ad, Fuj,  Gr1, Gold0, Gold, Kob, Pet}). 
The study of Picard's problem on a manifold may have originated from 
 the celebrated work due to S. T. Yau \cite{Yau} 
  on   Liouville’s theorem,  which  
   proves   the Liouville's   property of  harmonic functions on   
  Ricci non-negatively   curved  manifolds. Yau's work  inspires  people  to think about        
   the following    Picard's problem.     Let $(M, g)$ be a complete noncompact K\"ahler manifold with   non-negative Ricci curvature 
     (see   examples  for such manifolds    in Sha-Yang \cite{Y-S} and Tian-Yau \cite{Y-T, T-Y}). 
        
             The classical   Picard's problem  states     that   

\noindent\textbf{Picard's Problem.}   \emph{Is every meromorphic function  on $M$ necessarily  a  constant  if it omits $3$ distinct values$?$}

\noindent\textbf{Remark 1.}\ 
Let $p: \mathbb D\to\mathbb P^1(\mathbb C)\setminus\{0,1,\infty\}$ be the universal covering, where $\mathbb D$ 
is  the unit disc in $\mathbb C.$ Thanks to the Liouville’s theorem by  Yau, 
one may give  an affirmative  answer to the Picard's problem
   if  every  holomorphic mapping $f: M\to\mathbb P^1(\mathbb C)\setminus\{0,1,\infty\}$
 can  lift     to a holomorphic mapping $\tilde f: M\to\mathbb D$ by $p.$  
   Using  the mapping lifting theorem, 
      $f$ admits  a lifting  $\tilde f$  if and only if $$f_*\big(\pi_1(M, x_0)\big)\subseteq p_*\big(\pi_1(\mathbb D, \tilde y_0)\big)$$ for some   $x_0\in M$ 
         such   that 
  $p(\tilde y_0)=y_0=f(x_0).$ 
  Since $\mathbb D$ is  simply-connected,   the fundamental group   $\pi_1(\mathbb D, \tilde y_0)$ is trivial. 
  This implies  that $f_*(\pi_1(M, x_0))$ is
    necessarily a 
   trivial group 
   if  such a lifting $\tilde f$ admits.
         However,  $f_*(\pi_1(M, x_0))$ is   not   trivial  in   most cases. 
          Thus,  we cannot lift $f$ to $\tilde f$ in general. 

An  early  result  in this direction  could   be dated  back to 1970, 
S. Kobayashi \cite{Kob} obtained     
some  Picard-type theorems for  holomorphic mappings between certain complex
manifolds.  In particular,  he showed that 

\begin{tha}[Kobayashi, 1970] Assume that  there is a complex Lie group acting  on $M$
transitively.  Then,  
every  meromorphic function   on $M$  must be  a constant if  it omits $3$ distinct values.
\end{tha}
Kobayashi's result is subject   to certain  complex manifolds which  are  acted on 
transitively by a complex Lie group.  In 1975,  Goldberg-Ishihara-Petridis \cite{Gold}   
treated     a class of  locally flat   manifolds (see N. Petridis \cite{Pet} also). 
They  showed  that 
\begin{thb}[Goldberg-Ishihara-Petridis, 1975] If $M$ is locally flat, then  every  holomorphic mapping  $f: M\to \mathbb P^1(\mathbb C)\setminus\{0, 1, \infty\}$ of bounded dilatation must be  a constant.
\end{thb}
 
 It is known that   Nevanlinna theory  \cite{Nev} is a great 
 development   of    
  Picard's  
theorem, which    investigates     the  value distribution
  of  meromorphic mappings between   complex manifolds. 
For instance,    refer  to    L. Ahlfors \cite{ahlfors}, 
H. Cartan \cite{Cart},  Carlson-Griffths-King \cite{gri, gri1}, J. Noguchi \cite{Ng, No}, 
 E.  Nochka \cite{nochka},  M. Ru \cite{ru00, ru}, 
 B. Shiffman \cite{Shi},  B. Shabat \cite{Shabat}, F. Sakai \cite{Sa, Sa1}, 
  W. Stoll \cite{Stoll, Stoll1}, P. Vojta \cite{Voj},  H. Wu \cite{wu}, 
   and  also refer to     \cite{at0, at1, at, Dong3, Dong1, DY} and references therein.

In 2010,  A. Atsuji \cite{at1}  gave  an affirmative  answer 
 for  a class of  meromorphic functions with  slower  growth    
by introducing    a Nevanlinna-type theory based on heat diffusion, 
We   introduce  his main    work. 
Let $\alpha, \Delta$ stand for   the K\"ahler form  and   Laplace-Beltrami operator on  $M,$ respectively.  
Let $f: M\to\mathbb P^1(\mathbb C)$ be a meromorphic function.  
 We equip  $\mathbb P^1(\mathbb C)$ with  Fubini-Study metric $\omega_{FS}.$ 
The K\"ahlerness of $M$ indicates    that 
 the Hilbert-Schmidt norm    of    differential $df$ 
  with respect  to metrics  $\alpha, \omega_{FS}$ can be  expressed    as    
  $$\|df\|^2=4m\frac{f^*\omega_{FS}\wedge\alpha^{m-1}}{\alpha^m}=\Delta\log\|f\|^2.$$
\ \ \ \    He   showed that 
\begin{thc}[Atsuji, 2010] Let $f: M\to\mathbb P^1(\mathbb C)$ be a meromorphic function.   Assume that $f$ satisfies  the  growth condition 
$$
\int_1^\infty e^{-\epsilon t^2}dt\int_{B(t)}\|df\|^2dv<\infty
$$
for any $\epsilon>0,$ where $B(t)$ is the geodesic ball    centered at a fixed reference point $o$ with radius $t$ in $M.$ 
Then, $f$  must be a constant if it omits $3$ distinct values.
\end{thc}

 Atsuji's method   \cite{at1} (see \cite{Dong3} also) is  to use    
  stochastic calculus by Brownian motion. 
 Atsuji  successfully  introduced   Nevanlinna-type functions $\tilde T_f(t, \omega_{FS}),$ 
 $\tilde m_f(t, a)$ and $\tilde N_f(t, a),$  where  
  $t$ is the time of  the Brownian motion $X_t$ on $M.$ 
  By  computing    curvatures and using   It\^o's formula (see  \cite{Ni}), 
    he established an  analogue   of the  Nevanlinna theory. 
       Note    that   $M$ is  stochastically complete (i.e., $X_t$ is conservative)  due to     
   Grigor'yan's criterion (see \cite{Gy}). 
In order  to  make $\tilde T_f(t, \omega_{FS})$ and $\tilde N_f(t, a)$   meaningful, 
the  following   are necessary: 
 \begin{enumerate}
  \item[$\bullet$]    $\tilde T_f(t, \omega_{FS})<\infty$ for $t>0$;
  
     \item[$\bullet$]   $\tilde T_f(t, \omega_{FS})\to\infty$ as $t\to\infty$;    

    \item[$\bullet$] $\tilde N_f(t,a)=0$  if  $f$ omits $a.$
\end{enumerate}
 Hence, $f$ needs  to  satisfy  certain  growth assumption. 
That is why a growth condition  is added  to  Theorem C.

  In this  paper,  our   purpose is  to  study  the Picard's problem  using  the tool of Nevanlinna theory, i.e., 
  we   shall  
     extend     the classical     Nevanlinna theory   to   non-parabolic  complete K\"ahler manifolds with non-negative Ricci curvature, and apply  the Second Main Theorem to 
   establish  a   Picard's theorem. 
     Let us 
   first  recall the  basic  notions   of  the non-parabolicity and the Ricci curvature  of    Riemannian manifolds.  
     Let $M$ be a complete  Riemannian manifold.  One says that  $M$ is  non-parabolic, if  there exists  a positive global Green function for $M,$ and parabolic otherwise. 
Or equivalently,  in a viewpoint of geometric analysis,  $M$ is  called  non-parabolic if 
$$G(x,y)=2\int_0^\infty p(t,x, y)dt<\infty$$
for $x\not=y,$ and parabolic if this infinite integral is divergent, where $p(t,x,y)$ denotes    the transition density function of the  Brownian motion $X_t$ generated by the 
 Laplace-Beltrami operator $\Delta$ on $M,$ 
and which  is  also called the heat kernel of $M.$
Note that  when $M$ is non-parabolic,  $G(x,y)$  defines  the  unique  minimal positive  global Green function of $\Delta/2$ for $M.$  More details may refer to   \cite{Li-Tam, L-T-W, Va0, Va}.  
Let $R$ stand for   the Riemannian curvature tensor of $M.$ 
The Ricci curvature tensor of $M$ is defined by  
 $${\rm{Ric}}(X, Y)=\sum_{j=1}^{\dim M} R(X, e_j, e_j, Y)$$
for any $X, Y\in T_{x}M$ with $x\in M,$ where $\{e_1,\cdots, e_{\dim M}\}$ is an orthonormal basis of $T_{x}M.$
 The Ricci curvature  of $M$ at $x$ in a direction $X$ is  defined  by
 $${\rm{Ric}}(X)=\frac{{\rm{Ric}}(X, X)}{\|X\|^2}.$$ We say that $M$ is of   non-negative Ricci curvature at $x,$ if ${\rm{Ric}}(X)\geq0$ for each nonzero vector  $X\in T_{x}M;$ and that  $M$ is of   non-negative Ricci curvature, if 
  it is of  non-negative Ricci curvature at each  $x\in M.$ 
 Moreover, the sectional curvature of $M$ at $x$ along a section $\Pi$ spanned by $X, Y\in T_{x}M$ is defend by 
 $$K(\Pi)=\frac{R(X, Y, Y, X)}{\|X\wedge Y\|^2}.$$
From the definitions,  the sign of  the sectional curvature at a point determines the sign of the Ricci curvature at the same point.
 
      To facilitate the construction of  examples later for  manifolds considered in this paper, 
we shall provide some  criteria  of  non-parabolicity via the  volume growth. 
A  sharp necessary condition by N. Varopoulos \cite{Va1} states that if $M$ is non-parabolic, then 
\begin{equation}\label{cri}
\int_1^\infty\frac{tdt}{V(t)}<\infty,
\end{equation}
where $V(t)$ denotes    the Riemannian volume of a geodesic ball   centered at a  fixed  reference point  $o$ with radius $t$ in $M.$
However,   (\ref{cri}) is far from  sufficient,   and a counterexample was constructed  in \cite{Va1}.
The first major result for the sufficiency was due to N. Varopoulos \cite{Va} and Li-Yau \cite{Li-Yau}.
Based on  the   heat kernel estimates, they proved that if $M$ has non-negative Ricci curvature and (\ref{cri}) is satisfied, then $M$ is  non-parabolic.
So, for  $M$ with non-negative Ricci curvature, $M$ is non-parabolic if and only if (\ref{cri}) is satisfied.

Next,   assume that   $(M, g)$ is  a  non-parabolic complete noncompact  K\"ahler  manifold with non-negative Ricci curvature, of complex dimension $m.$
  Here, the  non-parabolicity of $M$ is meant  that  $M$ is non-parabolic as a Riemannian manifold.  
 The Chern-Ricci form $\mathscr R$ of $g$ on $M$ is defined  by 
  $$\mathscr R=-dd^c\log\det(g_{i\bar j})$$ with $dd^c=\sqrt{-1}\partial\overline{\partial}/(2\pi).$ Since  $M$ is K\"ahler,   ${\rm{Ric}}\geq0$ is equivalent to     $\mathscr R\geq0.$
   Our innovative   idea   in the present   paper     is to construct a family  $\{\Delta(r)\}_{r>0}$ of  exhaustive precompact    domains   for 
 $M$ by means of  the   minimal positive  global Green function  for $M,$   so   that 
  it  allows   us to  well define   Nevanlinna's functions on  $\Delta(r)$ 
     and then  establish   
   a Second Main Theorem by estimating  the gradient  of   Green function for $\Delta(r).$ 
     Applying       the Second Main  Theorem, 
  we     provide   a positive answer to the Picard's problem under certain    condition.

\vskip\baselineskip

 \subsection{Main results}~

Let  $G(o,x)$ be   the minimal positive  global Green function  of $\Delta/2$ for $M,$  
 in which $\Delta$ is the Laplace-Beltrami operator on $M,$
 and  $o$ is a fixed reference point in $M.$ 
  Using   Li-Yau's estimate  \cite{Li-Yau},  there are constants $A, B>0$ such that  
 $$A\int_{\rho(x)}^\infty\frac{tdt}{V(t)}\leq G(o,x)\leq B\int_{\rho(x)}^\infty\frac{tdt}{V(t)}$$
holds for all $x\in M$ and all $t>0,$   in which  $\rho(x)$ is  the Riemannian distance between   $x$ and $o,$ and  $V(t)$ stands for   
  the Riemannian volume of the geodesic ball   centered at $o$ with radius $t$ in $M.$  For $r>0,$  define    
  $$\Delta(r)=\left\{x\in M: \    G(o,x)>A\int_r^\infty\frac{tdt}{V(t)}\right\}.$$
  Then, we have $o\in\Delta(r)$ for all $r>0$ and  the family $\{\Delta(r)\}_{r>0}$ exhausts $M.$
  
  Let  $X$ be a complex projective manifold with $\dim_{\mathbb C} X\leq m,$ over which one  can put a Hermitian positive line bundle $(L, h)$ with  Chern form $c_1(L, h)>0.$
 We now take   a  divisor  $D\in|L|,$ where $|L|$ denotes  the complete linear system of $L.$ 
  Given a  meromorphic mapping $f: M\to X,$ we  have the Nevanlinna's functions $T_f(r, L), m_f(r,D), N_f(r,D), \overline{N}_f(r,D)$ on $\Delta(r)$ 
  (see   Section \ref{sec32}).  Let    $K_X$ be the canonical line bundle over $X.$ 
The characteristic function of $\mathscr R$ is defined by
$$T(r,\mathscr R)= \frac{\pi^m}{(m-1)!}\int_{\Delta(r)}g_r(o,x)\mathscr R\wedge\alpha^{m-1},$$
where $\alpha$ is the K\"ahler form of $M,$ and  $g_r(o,x)$ is   the Green function of $\Delta/2$ for $\Delta(r)$ with a pole at $o$ satisfying Dirichlet boundary condition.   Set
\begin{equation}\label{Hr}
H(r,\delta)=\frac{1}{r}\left(\frac{V(r)}{r}\right)^{1+\delta}\int_{r}^\infty\frac{tdt}{V(t)}. 
 \end{equation} 

 The   main result  is the following  Second Main Theorem. 

     \begin{theorem}[=Theorem \ref{main}]\label{main11}  
 Let $M$ be a non-parabolic  complete  non-compact  K\"ahler manifold with non-negative Ricci curvature. 
 Let $X$ be a complex projective manifold of complex dimension not greater than that  of $M.$
 Let $D\in|L|$ be a reduced divisor of simple normal crossing type,  where $L$ is a positive line bundle over $X.$ Let $f:M\rightarrow X$ be a differentiably non-degenerate meromorphic mapping.  Then  for any $\delta>0,$ there exists a subset $E_\delta\subseteq(0, \infty)$ of finite Lebesgue measure such that 
$$T_f(r,L)+T_f(r, K_X)+T(r, \mathscr R)\leq \overline N_f(r,D)+O\left(\log^+T_f(r,L)+\log H(r,\delta)\right)$$
holds for all $r>0$ outside $E_\delta,$ where $H(r,\delta)$ is given by $(\ref{Hr}).$ 
\end{theorem}

For a divisor  $D\in|L|,$  the simple defect  of $f$ with respect to $D$  is defined   by
$$ \bar\delta_f(D)=1-\limsup_{r\rightarrow\infty}\frac{\overline{N}_f(r,D)}{T_f(r,L)}.$$
Set 
$$\left[\frac{c_1(K_X^*)}{c_1(L)}\right]=\inf\left\{s\in\mathbb R: \ \eta< s\omega;  \  \   ^\exists\eta\in c_1(K^*_X), \  ^\exists\omega\in c_1(L)\right\}.$$

We obtain   a  defect relation: 
\begin{cor}[=Corollary \ref{dde}]\label{app1}  Assume the same conditions as in Theorem $\ref{main11}.$ Then 
$$\bar\delta_f(D)
\leq \left[\frac{c_1(K_X^*)}{c_1(L)}\right]-\liminf_{r\rightarrow\infty}\frac{T(r,\mathscr R)}{T_f(r,L)}\leq \left[\frac{c_1(K_X^*)}{c_1(L)}\right], 
$$
if one of the following conditions is satisfied$:$

\noindent $(i)$ $M$ satisfies the volume growth condition 
\begin{equation}\label{cond}
\lim_{r\to\infty}\frac{\log\left(\frac{V(r)}{r^2}\displaystyle\int_r^\infty\frac{tdt}{V(t)}\right)}{\log r}=0;
\end{equation}
$(ii)$ $f$ is  of non-polynomial type growth, i.e., 
$f$ satisfies the growth condition 
$$\lim_{r\to\infty}\frac{\log r}{T_f(r,L)}=0.$$
\end{cor}

\noindent\textbf{Remark 2.}\  Note that  $r^2=o(V(r))$ as $r\to\infty,$ due to  
$$\int_1^\infty\frac{tdt}{V(t)}<\infty.$$
The  condition (\ref{cond})  is  relaxed. For example, it holds when  $V(r)=O(r^\mu)$ with $\mu>2;$ and holds when   $V(r)=O(r^\mu\log^{\nu}r)$ with $\mu>2$   or  $\mu= 2$  and $\nu>1.$ 

As an application of Corollary \ref{app1},  we consider  the Picard's problem under certain condition below. 
Treat  the case when $X=\mathbb P^1(\mathbb C)$ with the point line bundle, i.e., the hyperplane line bundle $\mathscr O(1).$ 
Take  $D=a_1+\cdots+a_q,$ where $a_1,\cdots, a_q$ are $q$ distinct points in $\mathbb P^1(\mathbb C).$  
If we put $L=q\mathscr O(1),$ then  $D\in|L|.$ 
A basic and well-known fact states that 
$$\left[\frac{c_1\big(K_{\mathbb P^1(\mathbb C)}^*\big)}{c_1(\mathscr O(1))}\right]=2,$$
which yields that 
$$\sum_{j=1}^q\bar\delta_f(a_j)\leq 2$$
whenever   $(i)$ or $(ii)$ in Corollary \ref{app1} is satisfied.  
According to  the  arguments above, we conclude that 

\begin{cor}\label{dde1}    Let $M$ be a non-parabolic complete noncompact   K\"ahler manifold with non-negative Ricci curvature.   If  $M$ satisfies the volume growth condition $(\ref{cond}),$
then every  meromorphic  function   on $M$  reduces  to a constant if it omits $3$  distinct values. 
\end{cor}
 
\begin{cor}\label{dde2}    Let $M$ be a non-parabolic complete noncompact   K\"ahler manifold with non-negative Ricci curvature.  Then,  every  meromorphic  function   of  
non-polynomial type growth  on $M$  can omit $2$  distinct values at most. 
\end{cor}

To close the section, we finally  provide some typical examples for  $M$  where   Corollaries \ref{dde1} and \ref{dde2}  apply. 
According to (\ref{cri}), we first notice   that $M=\mathbb C^m$ with $m\geq2$ is a  non-parabolic complete noncompact   K\"ahler manifold
 with zero  Ricci curvature. Below, we  give two nontrivial examples. 

\noindent\textbf{Example 1.} For an integer $l\geq1,$ we equip  $\mathbb P^l(\mathbb C)$  with Fubini-Study metric $\omega_{FS}.$ Note that $\mathbb P^l(\mathbb C)$ is a complete K\"ahler manifold with positive  sectional curvature (and thus with 
positive Ricci curvature).
We consider the product manifold $M=\mathbb C^k\times\mathbb P^l(\mathbb C)$ with $k\geq2,$ which is equipped with the  product metric induced from $\omega_{FS}$ and  the standard  Euclidean metric on  $\mathbb C^k.$  Because  $\mathbb C^k$ is a complete noncompact K\"ahler manifold, so is $M.$
Since  $\mathbb C^k$ is flat, $M$ is clearly  Ricci non-negatively curved  under this product metric.  Again,  by the dimension condition that  $k\geq2,$
we deduce that $M$ is of  the volume growth:  $O(r^{2k})\geq O(r^{4})$ as $r\to\infty,$ 
which leads to that $(\ref{cri})$ is satisfied. Therefore, $M$ is  a non-parabolic complete noncompact 
  K\"ahler manifold with non-negative Ricci curvature. It is not difficult   to see that there exist  
   a lot of     nonconstant meromorphic functions on $M.$ For instance, we can take $f(z, \zeta)=g(z)+h(\zeta)$ or $f(z, \zeta)=g(z)h(\zeta),$  where   
$g(z)$ is any  meromorphic function on $\mathbb C^k$  and $h(\zeta)$ is any  rational function on $\mathbb P^l(\mathbb C).$

\noindent\textbf{Example 2.}  Treat  a product manifold $M=\mathbb C^k\times T^{l}$ with $k\geq2$ and $l\geq1,$ where  $T^{l}=\mathbb C^{l}/\Lambda$ denotes  the $l$-dimensional complex torus  equipped with the metric inherited from the standard  Euclidean metric on  $\mathbb C^l.$  
Note  that $M$ is a complete noncompact  K\"ahler manifold with zero Ricci curvature under the induced product metric.
 By $k\geq2,$ the similar argument as in Example 1 
also shows the non-parabolicity of  $M.$ The existence of nonconstant meromorphic functions on $M$ can  be  confirmed easily.  
For instance, for  $l=1,$     we can take $f(z, w)=g(z)+h(w)$ or $f(z, w)=g(z)h(w),$  
 where   
$g(z)$ is any  meromorphic function on $\mathbb C^k$ and $h(w)$ is any elliptic function with  a periodic  lattice $\Lambda$
on  $\mathbb C.$

More general, $M$ can be taken as the product of $M_1$ and $M_2,$ in which  $M_1$ is any complete compact K\"ahler manifold with non-negative Ricci curvature, and 
$M_2$ is a complete noncompact K\"ahler manifold with non-negative Ricci curvature and volume growth satisfying   (\ref{cri}). In particular,  $M=M_2$ satisfies the desired  conditions. 

  \section{Some Facts from  Geometric Analysis}

\subsection{Volume Comparison Theorem}~

A   space form  is  a  complete (simply-connected) Riemannian manifold with constant sectional curvature.
Let  $M^K$ denote   the $n$-dimensional   space form   with constant sectional curvature $K,$ and   let 
      $V(K, r)$ denote   the Riemannian volume of a geodesic ball with radius $r$ in $M^K.$
     Let  $(M, g)$ denote   a   
 complete Riemannian manifold of dimension $n$ with Ricci curvature  tensor ${\rm{Ric}}_M,$ 
and   
     let  
  $V(r)$  denote    the Riemannian 
  volume of  a geodesic ball centered at a fixed reference point $o$   with radius $r$ in $M.$

The well-known  volume comparison theorem by  Bishop-Gromov (see   \cite{B}) states  that  

\begin{theorem} If  ${\rm{Ric}}_M\geq(n-1)Kg,$  then the volume ratio $V(r)/V(K, r)$ is a non-increasing function  in $r>0,$ and which  tends to $1$ as $r\to0.$  Hence, 
we have 
$$V(r)\leq V(K, r)$$
holds for  all $r>0.$
\end{theorem}

When    $M^K=\mathbb R^n,$  we obtain:    

\begin{cor}\label{volume}  We have 
$$V(r)\leq \omega_n r^n$$
holds for  all $r>0,$ where $\omega_{n}$ denotes  the standard Euclidean volume of the unit ball in $\mathbb R^n.$
\end{cor}

When ${\rm{Ric}}_M\geq0,$   Calabi-Yau (see  \cite{S-Y})  showed  that 

\begin{theorem}\label{volume1}  Assume   that $M$ is noncompact.  If  ${\rm{Ric}}_M\geq0,$  then $M$ has an infinite volume. More precisely, for any $\epsilon_0>0,$ there exists a constant $c=c(o, n, \epsilon_0)>0$ such that 
$$V(r)\geq c r$$
holds for  all $r\geq \epsilon_0.$
\end{theorem}

\subsection{Estimates of Heat Kernels}~

Let $M$ be a complete Riemannian manifold of dimension $n,$ with  Laplace-Beltrami operator $\Delta.$ 
Fix a reference point $o\in M.$ 
We denote by  $\rho(x)$  the Riemannian distance function of $x$ from $o.$ 
The heat kernel $p(t, x, y)$ of  $M$  is  the minimal positive fundamental solution to the  following heat equation 
$$\Big(\Delta-\frac{\partial}{\partial t}\Big)u(t, x)=0.$$
\ \ \ \  Li-Yau \cite{Li-Yau} (see \cite{S-Y} also) obtained the  two-sided estimates     of  $p(t,o,x).$ 
\begin{theorem} Assume that $M$ has non-negative Ricci curvature.  Then for any $0<\epsilon<1,$ there exist constants $C_1=C_1(\epsilon, n)>0$ and $C_2=C_2(\epsilon, n)>0$
such that 
$$\frac{C_1}{V(\sqrt t)}e^{-\frac{\rho(x)^2}{(4-\epsilon)t}}\leq p(t, o,x)\leq \frac{C_2}{V(\sqrt t)}e^{-\frac{\rho(x)^2}{(4+\epsilon)t}}$$
holds for all $x\in M$ and all $t>0.$
\end{theorem}
 Set 
$$G(o, x)=2\int_0^\infty p(t, o,x)dt.$$
 This infinite  integral is  convergent if only if $M$ is non-parabolic. When $M$ is non-parabolic, 
 $G(o, x)$ is  the unique   minimal positive global Green function of $\Delta/2$ for $M$ with a pole at $o,$ i.e., 
$$\begin{cases}
 -\frac{1}{2}\Delta G(o,x)=\delta_o(x),  \ \ &    x\in M; \\
  G(o,x)>0, \ \  &  x\in\partial M; \\
  \displaystyle{\lim_{\rho(x)\to\infty}}G(o,x)=0.
\end{cases}$$
where $\delta_o$ is the Dirac's delta  function with a pole at $o.$ 

When $M$ is non-parabolic with non-negative Ricci curvature,  Li-Yau \cite{Li-Yau}  (see \cite{S-Y} also) gave  two-sided bounds of $G(o,x)$ as follows.  
 \begin{theorem}\label{kkk}  Assume that  $M$ has   non-negative Ricci curvature. If $M$ is non-parabolic,  then   there exist constants $A, B>0$ depending only on $n$ 
such that 
 $$A\int_{\rho(x)}^\infty\frac{tdt}{V(t)}\leq G(o,x)\leq B\int_{\rho(x)}^\infty\frac{tdt}{V(t)}$$
holds for all $x\in M.$
\end{theorem}

  \section{Nevanlinna's Functions and First Main Theorem}

Let  $(M, g)$ be   a   non-parabolic complete noncompact  K\"ahler manifold  with non-negative Ricci curvature,
  of complex dimension $m.$ 
  The  K\"ahler form $\alpha$ of $M$ associated to metric   $g$ is defined by 
 $$\alpha=\frac{\sqrt{-1}}{\pi}\sum_{i,j=1}^mg_{i\bar j}dz_i\wedge d\bar z_{j}.$$

    \subsection{Construction of $\Delta(r)$}~

 Fix a  reference point $o\in M.$ 
  Let   $V(r)$ be  the Riemannian volume of  $B(r),$ where  $B(r)$  denotes  the geodesic ball  centered at $o$ with radius  $r$ in $M.$   Note that  the non-parabolicity of $M$ implies that
 $$\int_1^\infty\frac{tdt}{V(t)}<\infty.$$
Thus, we have the unique minimal positive global Green function $G(o,x)$ of the half Laplace-Beltrami operator $\Delta/2$  for $M,$  
which can be written as 
$$G(o, x)=2\int_0^\infty p(t, o,x)dt,$$
where $p(t,o,x)$ is the heat kernel of $M.$
  Let $\rho(x)$ be the Riemannian distance function of $x$ from $o.$ Using Theorem \ref{kkk}, there exist constants $A, B>0$ such that  
 \begin{equation}\label{Gr}
 A\int_{\rho(x)}^\infty\frac{tdt}{V(t)}\leq G(o,x)\leq B\int_{\rho(x)}^\infty\frac{tdt}{V(t)}
 \end{equation}
holds for all $x\in M.$ 
 For $r>0,$ define   
 $$\Delta(r)=\left\{x\in M: \    G(o,x)>A\int_r^\infty\frac{tdt}{V(t)}\right\}.$$
  Since 
$$\lim_{x\to o}G(o,x)=\infty, \ \ \ \     \lim_{\rho(x)\to\infty}G(o,x)=0,$$
  one can conclude immediately    that  $\Delta(r)$ is a precompact  domain containing $o$  such  that $\overline{\Delta(r_1)}\subseteq\Delta(r_2)$ whenever $r_1<r_2$ and that 
 $$ \ \ \  \lim_{r\to0}\Delta(r)\to \emptyset, \ \ \ \   \lim_{r\to\infty}\Delta(r)=M.$$
Thus,  the family 
     $\{\Delta(r)\}_{r>0}$ exhausts $M,$ i.e.,  for any   sequence $\{r_n\}_{n=1}^\infty$ with    $0<r_1<r_2<\cdots\to \infty,$ 
     we have       
 $$\bigcup_{n=1}^\infty\Delta(r_n)=M, \ \ \ \     \emptyset\not=\Delta(r_1)\subseteq\overline{\Delta(r_1)}\subseteq\Delta(r_2)\subseteq\overline{\Delta(r_2)}\subseteq\cdots$$    
The boundary $\partial\Delta(r)$ of $\Delta(r)$ can be described  as
 $$\partial\Delta(r)=\left\{x\in M: \    G(o,x)=A\int_r^\infty\frac{tdt}{V(t)}\right\}.$$
  By  Sard's theorem,   $\partial\Delta(r)$  is smooth  for almost all $r>0.$  
    
     Set
 $$g_r(o,x)=G(o,x)-A\int_r^\infty\frac{tdt}{V(t)}.$$
Evidently,    $g_r(o,x)$ is  the positive Green function of $\Delta/2$ for $\Delta(r)$ with a pole at $o$ satisfying Dirichlet boundary condition, i.e., 
      $$\begin{cases}
 -\frac{1}{2}\Delta g_r(o,x)=\delta_o(x),  \ \ &    x\in\Delta(r); \\
 g_r(o,x)=0, \ \  &  x\in\partial\Delta(r),
\end{cases}$$
 where $\delta_o$ is the Dirac's delta  function with a pole at $o.$ 
Let  $\pi_r$ stand for  the harmonic measure  on $\partial\Delta(r)$ with respect to $o,$ defined by
  $$d\pi_r=\frac{1}{2}\frac{\partial g_r(o,x)}{\partial{\vec{\nu}}}d\sigma_r,$$
  where  $\partial/\partial \vec\nu$ is the inward  normal derivative on $\partial \Delta(r),$ and $d\sigma_{r}$ is the  induced Riemannian area element of 
$\partial \Delta(r).$

     \subsection{Nevanlinna's Functions}~\label{sec32}

   In what follows, we will introduce  Nevanlinna's functions.  Let $f: M\to X$ be a meromorphic mapping, where $X$ is a complex projective manifold.  
Let $(L, h)$ be a Hermitian holomorphic line bundle over $X,$ with the Chern form $$c_1(L,h):=-dd^c\log h,$$ where 
$$d=\partial+\overline{\partial}, \ \ \ \     d^c=\frac{\sqrt{-1}}{4\pi}\big(\overline{\partial}-\partial\big)$$ 
 so that $$dd^c=\frac{\sqrt{-1}}{2\pi}\partial\overline{\partial}.$$

Fix a divisor  $D\in|L|,$  where $|L|$ is the complete linear system of $L.$
Let $s_D$ be the canonical section  associated to $D,$ i.e., $s_D$ is a holomorphic section of $L$ over $X$ with zero divisor $D.$ 
The characteristic function, proximity function, counting function and simple counting function of $f$ are respectively defined  by 
   \begin{eqnarray*}
T_f(r, L)&=& -\frac{1}{4}\int_{\Delta(r)}g_r(o,x)\Delta\log(h\circ f)dv, \\
m_f(r,D)&=&\int_{\partial\Delta(r)}\log\frac{1}{\|s_D\circ f\|}d\pi_r, \\
N_f(r,D)&=& \frac{\pi^m}{(m-1)!}\int_{f^*D\cap\Delta(r)}g_r(o,x)\alpha^{m-1}, \\
\overline{N}_f(r, D)&=& \frac{\pi^m}{(m-1)!}\int_{f^{-1}(D)\cap\Delta(r)}g_r(o,x)\alpha^{m-1},
 \end{eqnarray*}
 where   $dv$ is the  Riemannian volume element of $M.$

Locally, write $s_D=\tilde s_De,$ in which  $e$ is a local holomorphic   frame  of $L$ and  $\tilde s_D$ is a holomorphic function. 
  Using  Poincar\'e-Lelong formula  (see, e.g., \cite{gri}),  we get  
    $$\left[D\right]= dd^c\big[\log|\tilde s_D|^2\big]$$
    in the sense of currents.   So, we obtain  the alternative expressions of $T_f(r,L)$ and $N_f(r,D)$  as follows 
    \begin{eqnarray*}
 T_f(r, L) &=& \frac{\pi^m}{(m-1)!}\int_{\Delta(r)}g_r(o,x)f^*c_1(L,h)\wedge\alpha^{m-1}  \\
 &=&   \frac{1}{4}\int_{\Delta(r)}g_r(o,x)\|df\|^2dv 
     \end{eqnarray*}
     and 
     \begin{eqnarray*}
\ \ \ \ \  \ \ \  \  N_f(r,D)&=& \frac{\pi^m}{(m-1)!}\int_{\Delta(r)}g_r(o,x)dd^c\big[\log|\tilde s_D\circ f|^2\big]\wedge\alpha^{m-1} \\
 &=& \frac{1}{4}\int_{\Delta(r)}g_r(o,x)\Delta\log|\tilde s_D\circ f|^2dv.
   \end{eqnarray*}

   \subsection{First Main Theorem}~
   
To establish the First Main Theorem of $f,$ we need   Green-Dynkin formula (see  \cite{at, Dong1, DY}) as follows. 
\begin{lemma}[Green-Dynkin formula]\label{dynkin} Let $\phi$ be a $\mathscr C^2$ function on  $M$ outside a polar set of singularities at most. Assume that $\phi(o)\not=\infty.$  Then
$$\int_{\partial \Delta(r)}\phi d\pi_{r}-\phi(o)=\frac{1}{2}\int_{\Delta(r)}g_r(o,x)\Delta \phi dv.$$
\end{lemma}
 Assume   that $f(o)\not\in{\rm{Supp}}D.$ Apply  Green-Dynkin formula to $\log\|s_D\circ f\|,$  we are led to  
  \begin{eqnarray*}
  && m_f(r,D)-\log\frac{1}{\|s_D\circ f(o)\|} \\
&=& \frac{1}{2}\int_{\Delta(r)}g_r(o,x)\Delta\log\frac{1}{\|s_D\circ f\|}dv     \\
&=& -\frac{1}{4}\int_{\Delta(r)}g_r(o,x)\Delta\log(h\circ f)dv-\frac{1}{4}\int_{\Delta(r)}g_r(o,x)\Delta\log|\tilde s_D\circ f|^2dv  \\
&=& T_f(r,L)-N_f(r,D). 
 \end{eqnarray*}

To conclude, we obtain:  
\begin{theorem}[First Main Theorem]  Assume   that $f(o)\not\in{\rm{Supp}}D.$ Then
$$T_f(r,L)+\log\frac{1}{\|s_D\circ f(o)\|}=m_f(r,D)+N_f(r,D).$$
\end{theorem}

 \section{Gradient Estimates of Green Functions}
 
   Let $M$ be a non-parabolic complete  noncompact  K\"ahler manifold  with non-negative Ricci curvature, of complex dimension $m.$ 

To give a gradient estimate of $g_r(o,x),$  we need to introduce some   lemmas as follows. 
   \begin{lemma}\label{thm4} We have
 $$g_r(o,x)=A\int_t^r \frac{sds}{V(s)}, \ \ \ \    x\in\partial\Delta(t)$$
holds  for all $0<t\leq r,$ where $A$ is given by $(\ref{Gr}).$
  \end{lemma}
 \begin{proof}    According to the definition of Green function for $\Delta(r),$  it is immediate that  for $0<t\leq r$
  \begin{eqnarray*} 
g_r(o,x)&=& G(o,x)-A\int_{r}^\infty\frac{tdt}{V(t)}  \\
&=& G(o,x)-A\int_{t}^\infty\frac{sds}{V(s)}ds +A\int_{t}^r\frac{sds}{V(s)}  \\
 &=& g_t(o,x)+A\int_{t}^r\frac{sds}{V(s)}.   
   \end{eqnarray*}
 Since $g_t(o,x)=0$ for $x\in\partial\Delta(t),$ 
  we obtain 
 $$g_r(o,x)=A\int_t^r \frac{sds}{V(s)}, \ \ \ \   x\in\partial\Delta(t).$$
 \end{proof} 
 
 Let $\nabla$ denote the gradient operator on any Riemannian manifold.  Cheng-Yau \cite{C-Y} proved the following theorem.

 \begin{lemma}\label{CY}
 Let $N$ be a complete Riemannian manifold of dimension $n\geq2.$     Let   $B(x_0, r)$ be any geodesic ball  centered at $x_0$  with radius $r$  in $N.$
 Then, there exists  a constant $C_n>0$ depending only on $n$  such that 
 $$\frac{\|\nabla u(x)\|}{u(x)}\leq \frac{C_n r^2}{r^2-d(x_0, x)^2}\left(|\kappa(r)|+\frac{1}{d(x_0, x)}\right)$$
 holds for any  non-negative  harmonic function $u$ on $B(x_0, r),$ 
 where $d(x_0, x)$ is the Riemannian distance between $x_0$ and $x,$ and $\kappa(r)$ is the lower bound of  Ricci curvature of  $B(x_0, r).$
 \end{lemma}
  
  We  obtain an upper estimate  of  $\|\nabla g_r(o,x)\|$ as follows. 
   
 \begin{theorem}\label{hh} There exists a constant  $c_1>0$  independent of $r$ such that 
 $$\|\nabla g_r(o,x)\|\leq \frac{c_1}{r}\int_{r}^\infty\frac{tdt}{V(t)}, \ \ \ \     x\in\partial\Delta(r).$$
\end{theorem}
\begin{proof} By the curvature assumption,   ${\rm{Ric}}_M\geq0.$ Thus,  it yields   from Lemma \ref{CY} (letting $r\to\infty$) and (\ref{Gr})     that  (see Remark 5 in \cite{TS} also)
  \begin{eqnarray*}
\|\nabla G(o, x)\|\leq\frac{c_0}{\rho(x)}G(o,x) 
\leq \frac{c_0B}{\rho(x)}\int_{\rho(x)}^\infty\frac{tdt}{V(t)}
  \end{eqnarray*}
 for some large  constant $c_0>0$ which depends  only on  the dimension  $m.$  By  (\ref{Gr})   again 
 $$\int_{\rho(x)}^\infty\frac{tdt}{V(t)}\leq \int_r^\infty \frac{tdt}{V(t)}, \ \ \  \  x\in \partial\Delta(r),$$
 which  gives    
 $\rho(x)\geq r$ for  $x\in\partial\Delta(r).$
Set $c_1=c_0B.$  Then, we conclude that 
$$\|\nabla g_r(o, x)\|=\|\nabla G(o, x)\|\leq \frac{c_1}{r}\int_{r}^\infty\frac{tdt}{V(t)},  \ \ \ \      x\in\partial\Delta(r).$$
  \end{proof}
   
   By
   $$\frac{\partial g_r(o,x)}{\partial{\vec{\nu}}}=\|\nabla g_r(o, x)\|, \ \ \ \      x\in\partial\Delta(r),$$
we can derive    an  upper estimate of  $d\pi_r$  as follows. 
  \begin{cor}\label{bbb}  There exists a constant  $c_2>0$ independent of $r$  such that 
   $$d\pi_r\leq \frac{c_2}{r}\int_{r}^\infty\frac{tdt}{V(t)}d\sigma_r,$$
 where  $d\sigma_{r}$ is the induced Riemannian area element of 
$\partial \Delta(r).$
\end{cor}

 \section{Calculus Lemma and Logarithmic Derivative Lemma}
 
   Let $(M,g)$ be a non-parabolic complete  noncompact   K\"ahler manifold  with non-negative Ricci curvature, of complex dimension $m.$ 

 \subsection{Calculus Lemma}~
 
 We need the following Borel's growth  lemma (see  \cite{No, ru}). 
 \begin{lemma}[Borel's Growth Lemma]\label{} Let $u\geq0$ be a non-decreasing  function  on $(r_0, \infty)$ with $r_0\geq0.$ Then for any $\delta>0,$ there exists a subset $E_\delta\subseteq(r_0,\infty)$
 of finite Lebesgue measure such that  
 $$u'(r)\leq u(r)^{1+\delta}$$
 holds for all $r>r_0$ outside $E_{\delta}.$  
 \end{lemma}
 \begin{proof} The conclusion is clearly true  for $u\equiv0.$ Next, we assume that $u\not\equiv0.$
 Since $u\geq0$ is  non-decreasing,  there is a number $r_1>r_0$ such that $u(r_1)>0.$  Using the non-decreasing property of $u,$   the  limit
 $\eta=\lim_{r\to\infty}u(r)$
exists or $\eta=\infty.$   If $\eta=\infty,$ then $\eta^{-1}=0.$ 
 Set  
 $$E_\delta=\left\{r\in(r_0,\infty): \  u'(r)>u(r)^{1+\delta}\right\}.$$
Note that  $u'(r)$ exists for  almost all  $r\in(r_0, \infty).$  Then, we have  
   \begin{eqnarray*}
 \int_{E_\delta}dr 
 &\leq& \int_{r_0}^{r_1}dr+\int_{r_1}^\infty\frac{u'(r)}{u(r)^{1+\delta}}dr \\
 &=&\frac{1}{\delta u(r_1)^\delta}-\frac{1}{\delta \eta^\delta}+r_1-r_0 \\
 &<&\infty.
    \end{eqnarray*}
 This completes the proof. 
 \end{proof}

We establish  the following  Calculus Lemma. 

 \begin{theorem}[Calculus Lemma]\label{calculus}
Let $k\geq0$ be a locally integrable function on $M.$ Assume that $k$ is locally bounded at $o.$ Then there is a  constant $C>0$ independent of $r$ such that
 for any $\delta>0,$ there exists  a subset $E_{\delta}\subseteq(0,\infty)$ of finite Lebesgue measure such that
$$\int_{\partial\Delta(r)}kd\pi_r\leq CH(r,\delta)\bigg(\int_{\Delta(r)}g_r(o,x)kdv\bigg)^{(1+\delta)^2}$$
holds for all $r>0$ outside $E_{\delta},$ where $H(r,\delta)$ is given by $(\ref{Hr}).$  
\end{theorem}
 
  \begin{proof} 
 Invoking  Lemma  \ref{thm4}, we have   
   \begin{eqnarray*}
 \int_{\Delta(r)}g_r(o,x)kdv 
 &=&\int_0^rdt \int_{\partial\Delta(t)}g_r(o,x)kd\sigma_t \\
    &=&  A\int_0^r\left(\int_{t}^r\frac{sds}{V(s)}\right)dt \int_{\partial\Delta(t)}kd\sigma_t.
   \end{eqnarray*}
 Set 
 $$\Lambda(r)=A\int_0^{r}\left(\int_{t}^{r}\frac{sds}{V(s)}\right)dt \int_{\partial\Delta(t)}kd\sigma_t.$$
 A simple computation leads to  
 $$\Lambda'(r)=\frac{d\Lambda(r)}{dr}=\frac{Ar}{V(r)}\int_0^rdt\int_{\partial\Delta(t)}kd\sigma_t.$$
 In further, we have
 $$\frac{d}{dr}\left(\frac{V(r)\Lambda'(r)}{r}\right)=A\int_{\partial\Delta(r)}kd\sigma_r.$$
Employing   Borel's growth lemma  to the left hand side of the above equality  twice:  one is to $V(r)\Lambda'(r)/r$ and  another   is to 
 $\Lambda'(r),$ then  we conclude   that for any  $\delta>0,$   there exists a subset $E_\delta\subseteq(0,\infty)$ of finite Lebesgue measure such that 
$$ \int_{\partial\Delta(r)}kd\sigma_r\leq \frac{1}{A}\left(\frac{V(r)}{r}\right)^{1+\delta}\Lambda(r)^{(1+\delta)^2}
$$ holds for all $r>0$ outside $E_\delta.$  
  On the other hand,  Corollary \ref{bbb} implies that there exists a constant $c>0$ independent of $r$ such that
  $$d\pi_r\leq \frac{c}{r}\int_{r}^\infty\frac{tdt}{V(t)}d\sigma_r.$$
Set $C=c/A.$ Combining the above, we have 
   \begin{eqnarray*}
 \int_{\partial\Delta(r)}kd\pi_r &\leq& \frac{c}{Ar}\left(\frac{V(r)}{r}\right)^{1+\delta}\int_{r}^\infty\frac{tdt}{V(t)}\Lambda(r)^{(1+\delta)^2}\\ 
 &=&CH(r,\delta)\Lambda(r)^{(1+\delta)^2}.
   \end{eqnarray*}
 holds for all $r>0$ outside $E_\delta.$  The proof is completed. 
 \end{proof}

 \vskip\baselineskip

  \subsection{Logarithmic Derivative  Lemma}~

Let $\psi$ be a meromorphic function on $M.$
The norm of the gradient $\nabla\psi$  is defined by
$$\|\nabla\psi\|^2=2\sum_{i,j=1}^m g^{i\overline j}\frac{\partial\psi}{\partial z_i}\overline{\frac{\partial \psi}{\partial  z_j}}$$
in a local holomorphic  coordinate $z=(z_1,\cdots,z_m),$ where $(g^{i\overline{j}})$ is the inverse of $(g_{i\overline{j}}).$
Define the Nevanlinna's characteristic function of $\psi$ by  
$$T(r,\psi)=m(r,\psi)+N(r,\psi),$$
where    
\begin{eqnarray*}
m(r,\psi)&=&\int_{\partial\Delta(r)}\log^+|\psi|d\pi_r, \\
N(r,\psi)&=& \frac{\pi^m}{(m-1)!}\int_{\psi^*\infty\cap \Delta(r)}g_r(o,x)\alpha^{m-1}.
 \end{eqnarray*}
For  any $a\in\mathbb C,$ it is not hard  to deduce   that 
 \begin{equation*}\label{goed}
T\Big(r,\frac{1}{\psi-a}\Big)= T(r,\psi)+O(1).
 \end{equation*}
 \ \ \ \  Put a singular metric on $\mathbb P^1(\mathbb C)$:  
$$\Psi=\frac{1}{|\zeta|^2(1+\log^2|\zeta|)}\frac{\sqrt{-1}}{4\pi^2}d\zeta\wedge d\bar \zeta, \ \ \ \      \int_{\mathbb P^1(\mathbb C)}\Psi=1.$$
\begin{lemma}\label{oo12} We have
$$\frac{1}{4\pi}\int_{\Delta(r)}g_r(o,x)\frac{\|\nabla\psi\|^2}{|\psi|^2(1+\log^2|\psi|)}dv\leq T(r,\psi)+O(1).$$
\end{lemma}
\begin{proof}  A direct  computation   gives   
$$\frac{\|\nabla\psi\|^2}{|\psi|^2(1+\log^2|\psi|)}=4m\pi\frac{\psi^*\Psi\wedge\alpha^{m-1}}{\alpha^m}.$$
Whence,  we conclude  from 
 Fubini's theorem that 
\begin{eqnarray*}
&& \frac{1}{4\pi}\int_{\Delta(r)}g_r(o,x)\frac{\|\nabla\psi\|^2}{|\psi|^2(1+\log^2|\psi|)}dv \\ 
&=&m\int_{\Delta(r)}g_r(o,x)\frac{\psi^*\Psi\wedge\alpha^{m-1}}{\alpha^m}dv  \\
&=&\frac{\pi^m}{(m-1)!}\int_{\mathbb P^1(\mathbb C)}\Psi(\zeta)\int_{\psi^*\zeta\cap \Delta(r)}g_r(o,x)\alpha^{m-1} \\
&=&\int_{\mathbb P^1(\mathbb C)}N\Big(r, \frac{1}{\psi-\zeta}\Big)\Psi(\zeta)  \\
&\leq&\int_{\mathbb P^1(\mathbb C)}\big{(}T(r,\psi)+O(1)\big{)}\Psi \\
&=& T(r,\psi)+O(1). 
\end{eqnarray*}
\end{proof}

\begin{lemma}\label{999a}  Let
$\psi\not\equiv0$ be a  meromorphic function on  $M.$  Then for any $\delta>0,$ there exists a subset 
 $E_\delta\subseteq(0,\infty)$ of finite Lebesgue measure such that
  \begin{eqnarray*}
 &&\int_{\partial\Delta(r)}\log^+\frac{\|\nabla\psi\|^2}{|\psi|^2(1+\log^2|\psi|)}d\pi_r \\
  &\leq& (1+\delta)^2\log^+ T(r,\psi)+\log H(r,\delta)+O(1)
\end{eqnarray*}
 holds for all $r>0$ outside  $E_\delta,$   where $H(r,\delta)$ is given by $(\ref{Hr}).$   
\end{lemma}
\begin{proof} Since  $\pi_r$  is a harmonic measure on $\partial\Delta(r),$ 
it  is   a probability measure on  $\partial\Delta(r).$  Using the concavity of $``\log",$ we deduce that       
\begin{eqnarray*}
&& \int_{\partial\Delta(r)}\log^+\frac{\|\nabla\psi\|^2}{|\psi|^2(1+\log^2|\psi|)}d\pi_r \\
   &\leq&   \log\int_{\partial\Delta(r)}\bigg(1+\frac{\|\nabla\psi\|^2}{|\psi|^2(1+\log^2|\psi|)}\bigg)d\pi_r \\
    &\leq&  \log^+\int_{\partial\Delta(r)}\frac{\|\nabla\psi\|^2}{|\psi|^2(1+\log^2|\psi|)}d\pi_r+O(1). \nonumber
\end{eqnarray*}
By  this with  Theorem \ref{calculus} and Lemma \ref{oo12}, we conclude  that  for any $\delta>0,$ there exists a subset 
 $E_\delta\subseteq(0,\infty)$ of finite Lebesgue measure such that 
\begin{eqnarray*}
   && \log^+\int_{\partial\Delta(r)}\frac{\|\nabla\psi\|^2}{|\psi|^2(1+\log^2|\psi|)}d\pi_r\\
   &\leq& (1+\delta)^2 \log^+\int_{\Delta(r)}g_r(o,x)\frac{\|\nabla\psi\|^2}{|\psi|^2(1+\log^2|\psi|)}dv+\log H(r,\delta)+O(1) \\
   &\leq& (1+\delta)^2 \log^+T(r,\psi)+\log H(r,\delta)+O(1)
\end{eqnarray*}
holds  for all $r>0$ outside  $E_\delta.$  
\end{proof}
Define
$$m\left(r,\frac{\|\nabla\psi\|}{|\psi|}\right)=\int_{\partial\Delta(r)}\log^+\frac{\|\nabla\psi\|}{|\psi|}d\pi_r.$$
\ \ \ \  With  the above  preparations,  we are  to  prove the following  Logarithmic Derivative  Lemma. 
\begin{theorem}[Logarithmic Derivative  Lemma]\label{log1} Let
$\psi\not\equiv0$ be a  meromorphic function on  $M.$   Then for any $\delta>0,$ there exists a  subset  $E_\delta\subseteq(0,\infty)$ of  finite Lebesgue measure such that 
\begin{eqnarray*}
   m\Big(r,\frac{\|\nabla\psi\|}{|\psi|}\Big)&\leq& \frac{2+(1+\delta)^2}{2}\log^+ T(r,\psi)+\frac{1}{2}\log H(r,\delta)
\end{eqnarray*}
 holds for all $r>0$ outside  $E_\delta,$  where $H(r,\delta)$ is given by $(\ref{Hr}).$  
\end{theorem}
\begin{proof}  The concavity of $``\log"$ gives  that 
\begin{eqnarray*}
    m\left(r,\frac{\|\nabla\psi\|}{|\psi|}\right)  
   &\leq& \frac{1}{2}\int_{\partial\Delta(r)}\log^+\frac{\|\nabla\psi\|^2}{|\psi|^2(1+\log^2|\psi|)}d\pi_r \ \ \  \  \    \   \  \  \    \    \   \   \\ 
 &&   +\frac{1}{2}\int_{\partial\Delta(r)}\log\left(1+\log^2|\psi|\right)d\pi_r \\
  &=& \frac{1}{2}\int_{\partial\Delta(r)}\log^+\frac{\|\nabla\psi\|^2}{|\psi|^2(1+\log^2|\psi|)}d\pi_r \\
   && +\frac{1}{2}\int_{\partial\Delta(r)}\log\bigg(1+\Big{(}\log^+|\psi|+\log^+\frac{1}{|\psi|}\Big{)}^2\bigg)d\pi_r  \\
    &\leq&  \frac{1}{2}\int_{\partial\Delta(r)}\log^+\frac{\|\nabla\psi\|^2}{|\psi|^2(1+\log^2|\psi|)}d\pi_r \\
   && +\log\int_{\partial\Delta(r)}\Big{(}\log^+|\psi|+\log^+\frac{1}{|\psi|}\Big{)}d\pi_r +O(1)  \\
   &\leq& \frac{1}{2}\int_{\partial\Delta(r)}\log^+\frac{\|\nabla\psi\|^2}{|\psi|^2(1+\log^2|\psi|)}d\pi_r+\log^+T(r,\psi)+O(1).
\end{eqnarray*}
By Lemma \ref{999a},  we have for any $\delta>0,$ there exists a subset 
 $E_\delta\subseteq(0,\infty)$ of finite Lebesgue measure such that 
 \begin{equation*}
 m\left(r,\frac{\|\nabla\psi\|}{|\psi|}\right)  \leq \frac{2+(1+\delta)^2}{2}\log^+ T(r,\psi)+\frac{\delta}{2}\log H(r,\delta)
\end{equation*}
 holds for all $r>0$ outside  $E_\delta,$  where $H(r,\delta)$ is given by $(\ref{Hr}).$   
  \end{proof}

  \section{Second Main Theorem and Defect Relation}

   Let $(M, g)$ be a non-parabolic  complete noncompact  K\"ahler manifold  with non-negative Ricci curvature, of complex dimension $m.$
Let $X$ be a complex projective manifold of complex dimension not greater than $m.$ Over $X,$ one can put a positive holomorphic line bundle $(L, h).$ Let $D\in|L|$ be a reduced divisor of simple normal crossing type.  
      Recall that  
$$T(r,\mathscr R)= \frac{\pi^m}{(m-1)!}\int_{\Delta(r)}g_r(o,x)\mathscr R\wedge\alpha^{m-1},$$
where $\mathscr R=-dd^c\log\det(g_{i\bar j})$ is the Chern-Ricci form of $g$ on $M.$  

Let $Z=\sum_j\mu_jZ_j$ be a divisor, where  $Z_j^,s$ are prime divisors.  The reduced form of $Z$ is
 defined by ${\rm{Red}}(Z):=\sum_jZ_j.$

Write $D=D_1+\cdots+D_q$ as  the irreducible decomposition of $D.$
Equiping every holomorphic line bundle $\mathscr O(D_j)$ with a Hermitian
metric $h_j$ such that it  induces  the  Hermitian metric $h=h_1\otimes\cdots\otimes h_q$ on $L.$ 
Pick $s_j\in H^0(X, \mathscr O(D_j)$
such that  $(s_j)=D_j$ with $\|s_j\|<1.$
On $X,$ we define a singular volume form
$$ \Phi=\frac{\Omega}{\prod_{j=1}^q\|s_j\|^2}, \ \ \ \    \Omega=\wedge^nc_1(L,h).$$ 
\ \ \ \   Let $f: M\to X$ be a differentiably non-degenerate meromorphic mapping.
Set
$$f^*\Phi\wedge\alpha^{m-n}=\xi\alpha^m.$$
It is clear that
$$\alpha^m=m!\det(g_{i\bar j})\bigwedge_{j=1}^m\frac{\sqrt{-1}}{\pi}dz_j\wedge d\bar z_j.$$
 In further, we have   
$$dd^c[\log\xi]\geq f^*c_1(L, h_L)-f^*{\rm{Ric}}(\Omega)+\mathscr{R}-[{\rm{Red}}(f^*D)]$$
in the sense of currents.    
Thus, it  yields   that
\begin{eqnarray}\label{5q}
&& \frac{1}{4}\int_{\Delta(r)}g_r(o,x)\Delta\log\xi dv \\
&\geq& T_{f}(r,L)+T_{f}(r,K_X)+T(r,\mathscr{R})-\overline{N}_{f}(r,D).  \nonumber
\end{eqnarray}

We give  the main theorem  in this paper, i.e.,   the  Second Main Theorem as follows. 

\begin{theorem}[Second Main Theorem]\label{main}   Let $M$ be a non-parabolic complete noncompact    K\"ahler manifold with non-negative Ricci curvature. 
 Let $X$ be a complex projective manifold of complex dimension not greater than that  of $M.$
 Let $D\in|L|$ be a reduced divisor of simple normal crossing type,  where $L$ is a positive line bundle over $X.$ Let $f:M\rightarrow X$ be a differentiably non-degenerate meromorphic mapping. Then  for any $\delta>0,$ there exists a subset $E_\delta\subseteq(0, \infty)$ of finite Lebesgue measure such that 
$$T_f(r,L)+T_f(r, K_X)+T(r, \mathscr R)\leq \overline N_f(r,D)+O\left(\log^+T_f(r,L)+ \log H(r,\delta)\right)$$
holds for all $r>0$ outside $E_\delta,$  where $H(r,\delta)$ is given by $(\ref{Hr}).$   
\end{theorem}

\begin{proof}   
  Since $D$  has the     simple normal crossing type,   there exist 
  a finite  open covering $\{U_\lambda\}$ of $X$ and finitely many   rational functions
$w_{\lambda1},\cdots,w_{\lambda n}$ on $X$ for each $\lambda,$  such that $w_{\lambda1},\cdots, w_{\lambda n}$ are holomorphic on $U_\lambda$  with    
\begin{eqnarray*}
  dw_{\lambda1}\wedge\cdots\wedge dw_{\lambda n}(x)\neq0, & & \  \ ^\forall x\in U_{\lambda}; \\
  D\cap U_{\lambda}=\big{\{}w_{\lambda1}\cdots w_{\lambda h_\lambda}=0\big{\}}, && \  \ ^\exists h_{\lambda}\leq n.
\end{eqnarray*}
In addition, we  can  require that  $\mathscr O(D_j)|_{U_\lambda}\cong U_\lambda\times \mathbb C$ for all 
$\lambda,j.$ On  $U_\lambda,$   write 
$$\Phi=\frac{e_\lambda}{|w_{\lambda1}|^2\cdots|w_{\lambda h_{\lambda}}|^2}
\bigwedge_{k=1}^n\frac{\sqrt{-1}}{\pi}dw_{\lambda k}\wedge d\bar w_{\lambda k},$$
where  $e_\lambda$ is a  positive smooth function on $U_\lambda.$  
Let  $\{\phi_\lambda\}$ be a partition of the unity subordinate to $\{U_\lambda\}.$ Set 
$$\Phi_\lambda=\frac{\phi_\lambda e_\lambda}{|w_{\lambda1}|^2\cdots|w_{\lambda h_{\lambda}}|^2}
\bigwedge_{k=1}^n\frac{\sqrt{-1}}{\pi}dw_{\lambda k}\wedge d\bar w_{\lambda k}.$$
 Again, put $f_{\lambda k}=w_{\lambda k}\circ f.$  On  $f^{-1}(U_\lambda),$ we have  
\begin{eqnarray*}
 f^*\Phi_\lambda&=&
   \frac{\phi_{\lambda}\circ f\cdot e_\lambda\circ f}{|f_{\lambda1}|^2\cdots|f_{\lambda h_{\lambda}}|^2}
   \bigwedge_{k=1}^n\frac{\sqrt{-1}}{\pi}df_{\lambda k}\wedge d\bar f_{\lambda k} \\
   &=& \phi_{\lambda}\circ f\cdot e_\lambda\circ f\sum_{1\leq i_1\not=\cdots\not= i_n\leq m}
   \frac{\Big|\frac{\partial f_{\lambda1}}{\partial z_{i_1}}\Big|^2}{|f_{\lambda 1}|^2}\cdots 
   \frac{\Big|\frac{\partial f_{\lambda h_\lambda}}{\partial z_{i_{h_\lambda}}}\Big|^2}{|f_{\lambda h_\lambda}|^2}
   \left|\frac{\partial f_{\lambda (h_\lambda+1)}}{\partial z_{i_{h_\lambda+1}}}\right|^2 \\
 &&  \cdots\left|\frac{\partial f_{\lambda n}}{\partial z_{i_{n}}}\right|^2 
    \Big(\frac{\sqrt{-1}}{\pi}\Big)^ndz_{i_1}\wedge d\bar z_{i_1}\wedge\cdots\wedge dz_{i_n}\wedge d\bar z_{i_n}.
\end{eqnarray*}
 Fix any  $x_0\in M.$ Take a local holomorphic   coordinate $z=(z_1,\cdots,z_m)$ near $x_0$ and a local  holomorphic   coordinate 
  $\zeta=(\zeta_1,\cdots,\zeta_n)$ 
near $f(x_0)$ such that
$$\alpha|_{x_0}=\frac{\sqrt{-1}}{\pi}\sum_{j=1}^m dz_j\wedge d\bar{z}_j$$
and
$$c_1(L, h)\big|_{f(x_0)}=\frac{\sqrt{-1}}{\pi}\sum_{j=1}^n d\zeta_j\wedge d\bar{\zeta}_j.$$
 Set   
$f^*\Phi_\lambda\wedge\alpha^{m-n}=\xi_\lambda\alpha^m.$
Then, we have  $\xi=\sum_\lambda \xi_\lambda$ and   
 \begin{eqnarray*}
&& \xi_\lambda\big|_{x_0}  \\
&=& \phi_{\lambda}\circ f\cdot e_\lambda\circ f\sum_{1\leq i_1\not=\cdots\not= i_n\leq m}
   \frac{\Big|\frac{\partial f_{\lambda1}}{\partial z_{i_1}}\Big|^2}{|f_{\lambda 1}|^2}\cdots 
   \frac{\Big|\frac{\partial f_{\lambda h_\lambda}}{\partial z_{i_{h_\lambda}}}\Big|^2}{|f_{\lambda h_\lambda}|^2}
   \left|\frac{\partial f_{\lambda (h_\lambda+1)}}{\partial z_{i_{h_\lambda+1}}}\right|^2\cdots\left|\frac{\partial f_{\lambda n}}{\partial z_{i_{n}}}\right|^2 \\
   &\leq&  \phi_{\lambda}\circ f\cdot e_\lambda\circ f\sum_{1\leq i_1\not=\cdots\not= i_n\leq m}
    \frac{\big\|\nabla f_{\lambda1}\big\|^2}{|f_{\lambda 1}|^2}\cdots 
   \frac{\big\|\nabla f_{\lambda h_\lambda}\big\|^2}{|f_{\lambda h_\lambda}|^2} \\
   &&\cdot \big\|\nabla f_{\lambda(h_\lambda+1)}\big\|^2\cdots\big\|\nabla f_{\lambda n}\big\|^2.
\end{eqnarray*} 
Define a non-negative function $\varrho$ on $M$ by  
\begin{equation}\label{wer}
  f^*c_1(L, h)\wedge\alpha^{m-1}=\varrho\alpha^m.
\end{equation}
Moreover, put  $f_j=\zeta_j\circ f$ for $1\leq j\leq n.$  Then    
\begin{equation*}
    f^*c_1(L, h)\wedge\alpha^{m-1}\big|_{x_0}=\frac{(m-1)!}{2}\sum_{j=1}^m\big\|\nabla f_j\big\|^2\alpha^m, 
\end{equation*}
which yields that 
$$\varrho|_{x_0}=(m-1)!\sum_{i=1}^n\sum_{j=1}^m\Big|\frac{\partial f_i}{\partial z_j}\Big|^2 
=\frac{(m-1)!}{2}\sum_{j=1}^n\big\|\nabla f_j\big\|^2.$$
 Put together   the above, we are led to 
$$\xi_\lambda\leq 
\frac{ \phi_{\lambda}\circ f\cdot e_\lambda\circ f\cdot(2\varrho)^{n-h_\lambda}}{(m-1)!^{n-h_\lambda}}\sum_{1\leq i_1\not=\cdots\not= i_n\leq m}
    \frac{\big\|\nabla f_{\lambda1}\big\|^2}{|f_{\lambda 1}|^2}\cdots 
   \frac{\big\|\nabla f_{\lambda h_\lambda}\big\|^2}{|f_{\lambda h_\lambda}|^2}
$$
on $f^{-1}(U_\lambda).$
Since $\phi_\lambda\circ f\cdot e_\lambda\circ f$ is bounded on $M$ and 
$$\log^+\xi\leq \sum_\lambda\log^+\xi_\lambda+O(1),$$ 
 we obtain  
\begin{equation}\label{bbd}
   \log^+\xi\leq O\Big{(}\log^+\varrho+\sum_{k, \lambda}\log^+\frac{\|\nabla f_{\lambda k}\|}{|f_{\lambda k}|}+1\Big{)}. 
 \end{equation}  
 By  Lemma \ref{dynkin} 
\begin{equation}\label{pfirst}
 \frac{1}{2}\int_{\Delta(r)}g_r(o,x)\Delta\log\xi dv
=\int_{\partial\Delta(r)}\log\xi d\pi_r+O(1).
\end{equation}
Combining  (\ref{bbd}) with (\ref{pfirst}) and using Theorem \ref{log1}, we have 
\begin{eqnarray*}
&& \frac{1}{2}\int_{\Delta(r)}g_r(o,x)\Delta\log\xi dv \\
   &\leq& O\bigg{(}\sum_{k,\lambda}m\Big{(}r,\frac{\|\nabla f_{\lambda k}\|}{|f_{\lambda k}|}\Big{)}+\log^+\int_{\partial\Delta(r)}\varrho d\pi_r+1\bigg) \\
   &\leq& O\bigg{(}\sum_{k,\lambda}\log^+ T(r,f_{\lambda k})+\log^+\int_{\partial\Delta(r)}\varrho d\pi_r+1\bigg{)}
   \\
   &\leq& O\bigg{(}\log^+ T_{f}(r,L)+\log^+\int_{\partial\Delta(r)}\varrho d\pi_r+1\bigg{)}.   
   \end{eqnarray*}
Using Theorem \ref{calculus} and (\ref{wer}),  for any $\delta>0,$ there exists a subset $E_\delta\subseteq(0,\infty)$ of finite Lebesgue measure such that 
$$\log^+\int_{\partial\Delta(r)}\varrho d\pi_r 
   \leq (1+\delta)^{2}\log^+T_{f}(r,L)+\log H(r,\delta)+O(1)$$
   holds for all $r>0$ outside $E_\delta.$ 
Hence,  we conclude that  
\begin{equation*}\label{6q}
     \frac{1}{4}\int_{\Delta(r)}g_r(o,x)\Delta\log\xi dv 
 \leq O\left(\log^+ T_{f}(r,L)+\log H(r,\delta)\right)
\end{equation*}
for all $r>0$ outside $E_\delta.$  By this with   (\ref{5q}),  we prove the theorem.
\end{proof}

Recall that the  simple defect  of $f$ with respect to $D$  is  defined  by
 \begin{eqnarray*}
 \bar\delta_f(D)&=&1-\limsup_{r\rightarrow\infty}\frac{\overline{N}_f(r,D)}{T_f(r,L)}.
 \end{eqnarray*}
By  the First Main Theorem, we have    
$$0\leq\bar\delta_f(D)\leq 1.$$

By the use of  Theorem $\ref{main},$ it is immediate that 
\begin{cor}[Defect Relation]\label{dde}  Assume the same conditions as in Theorem $\ref{main}.$ Then 
$$\bar\delta_f(D)
\leq \left[\frac{c_1(K_X^*)}{c_1(L)}\right]-\liminf_{r\rightarrow\infty}\frac{T(r,\mathscr R)}{T_f(r,L)}\leq \left[\frac{c_1(K_X^*)}{c_1(L)}\right], 
$$
if one of the following conditions is satisfied$:$

\noindent $(i)$ $M$ satisfies the volume growth condition 
$$\lim_{r\to\infty}\frac{\log\left(\frac{V(r)}{r^2}\displaystyle\int_r^\infty\frac{tdt}{V(t)}\right)}{\log r}=0;$$
$(ii)$ $f$ is of non-polynomial type growth, i.e., $f$ satisfies the growth condition 
$$\lim_{r\to\infty}\frac{\log r}{T_f(r,L)}=0.$$
\end{cor}
\begin{proof}  We first show that the corollary  holds if  $(i)$ is satisfied.   By  Corollary \ref{volume}, we have $V(r)\leq \omega_{2m}r^{2m}.$   Thus, we are led to  that  
 \begin{eqnarray*} 
\frac{\log H(r,\delta)}{\log r}&=& \frac{\log\left(\frac{V(r)}{r^2}\displaystyle\int_r^\infty\frac{tdt}{V(t)}\right)}{\log r}+\frac{\delta\log\frac{V(r)}{r}}{\log r} \\
&\leq& \frac{\log\left(\frac{V(r)}{r^2}\displaystyle\int_r^\infty\frac{tdt}{V(t)}\right)}{\log r}+(2m-1)\omega_{2m}\delta. 
\end{eqnarray*}
Using the condition $(i),$  we obtain  
$$\limsup_{r\to\infty}\frac{\log H(r,\delta)}{\log r}\leq (2m-1)\omega_{2m}\delta.$$
Since each  non-constant meromorphic mapping has   growth    $O(\log r)$ at least,    we have the theorem proved if  $(i)$ is satisfied  due to  Theorem \ref{main}. 
For $r>1,$ we obtain 
 \begin{eqnarray*} 
H(r,\delta)&=& \frac{1}{r}\left(\frac{V(r)}{r}\right)^{1+\delta}\int_{r}^\infty\frac{tdt}{V(t)} \\
&\leq& \frac{1}{r}\left(\frac{\omega_{2m}r^{2m}}{r}\right)^{1+\delta}\int_{1}^\infty\frac{tdt}{V(t)}  \\
&\leq& c\omega_{2m}^{1+\delta}r^{(2m-1)(1+\delta)-1},
 \end{eqnarray*} 
 where 
 $$c=\int_{1}^\infty \frac{tdt}{V(t)}.$$
Thus, it yields from $(ii)$ that 
$$\lim_{r\to\infty}\frac{\log H(r,\delta)}{T_f(r,L)}=0.$$
This completes the proof. 
 \end{proof}

\vskip\baselineskip

\end{document}